\documentclass[12pt]{article}
\usepackage{amsmath,amssymb,amsthm}

\newtheorem{theorem}{Theorem}[subsection]
\newtheorem{thmy}{Theorem}

\newtheorem{corollary}[theorem]{Corollary}

\def\0{\leqno}

\title{How large a union of proper subgroups\\ of a finite group can be?}
\author{Marius T\u arn\u auceanu}
\date{September 19, 2019}

\begin{document}
\maketitle

\begin{abstract}
    Let $k$ be a positive integer and $G$ be a finite group that cannot be written as the union of
    $k$ proper subgroups. In this short note, we study the existence of a constant $c_k\in (0,1)$
    such that $|\cup_{i=1}^k H_i|\leq c_k|G|$, for all proper subgroups $H_1$, ..., $H_k$ of $G$.
\end{abstract}



\section{Introduction}

A well-known elementary result of group theory states that a group
cannot be written as the union of two proper subgroups. In Scorza
\cite{5} the groups which are the union of three proper subgroups
have been characterized. The analogous problems with three replaced
by four, five and six subgroups were solved by Cohn \cite{2}, while
the case of seven subgroups was studied by Tomkinson \cite{4}. Note
that an excellent survey on this topic is Bhargava \cite{1}.

Following Cohn's notation, for a group $G$ we will write
$\sigma(G)=n$ whenever $G$ is the union of $n$ proper subgroups,
but is not the union of any smaller number of proper subgroups. By
using this notation, we first recall the above mentioned results.

\bigskip\noindent{\bf Theorem 1.1.} {\it Let $G$ be a group. Then
\begin{itemize}
\item[\rm a)]
        $\sigma(G) \notin \{1, 2\}${\rm ;}
\item[\rm b)]
        $\sigma(G)=3$ if and only if $G$ has at least two subgroups of index $2$, or equivalently $G$ has a quotient isomorphic to $C_2 \times C_2${\rm ;}
\item[\rm c)]
        $\sigma(G)=4$ if and only if $\sigma(G) \neq 3$ and $G$ has at least two subgroups of index $3$, or equivalently $G$ has a quotient isomorphic to $C_3 \times C_3$ or $S_3${\rm ;}
\item[\rm d)]
        $\sigma(G)=5$ if and only if $\sigma(G) \notin \{3, 4\}$ and $G$ has a maximal subgroup of index $4$, or equivalently $G$ has a quotient isomorphic to $A_4${\rm ;}
\item[\rm e)]
        $\sigma(G)=6$ if and only if $\sigma(G) \notin \{3, 4, 5\}$ and $G$ has a quotient isomorphic to $C_5 \times C_5$, $D_5$ or to the group of order {\rm 20} with the following presentation $\langle a,b \mid a^5 = b^4=1, \hspace{1mm}ba=a^2b \rangle${\rm;}
\item[\rm f)]
        $\sigma(G)\neq 7$.
\end{itemize}}

Inspired by these results, the following problem is natural: given a positive integer $k$ and a finite group $G$ with $\sigma(G)>k$, how large can be a union of $k$ proper subgroups of $G$? In other words, is there a constant $c_k\in (0,1)$ such that $|\cup_{i=1}^k H_i|\leq c_k|G|$, for all proper subgroups $H_1$, ..., $H_k$ of $G$? Obviously, we have $c_1=\frac{1}{2}\,$. In the current note, we will prove that $c_2=\frac{3}{4}\,$ and $c_3=\frac{5}{6}\,$, the general case remaining open. Also, for an arbitrary $k$ we will formulate a conjecture about the maximum number of elements in a union of $k$ proper subgroups of $G$.

Most of our notation is standard and will usually not be repeated here. For basic notions and results on groups we refer the reader to \cite{3}.

\section{Main results}

\subsection{The case $k=2$}

In this case we can take $c_2=\frac{3}{4}\,$, as shows the following theorem.

\begin{theorem}
If $G$ is a finite group, then $|H_1\cup H_2|\leq\frac{3}{4}\,|G|$ for all proper subgroups $H_1$ and $H_2$ of $G$. Moreover,
the equality holds if and only if $H_1$ and $H_2$ are distinct maximal subgroups of index $2$.
\end{theorem}

\begin{proof}
Let $H_1$ and $H_2$ be two proper subgroups of $G$. Then
$$|H_1\cup H_2|=|H_1|+|H_2|-|H_1\cap H_2|=|H_1|+|H_2|-\frac{|H_1||H_2|}{|H_1H_2|}$$
$$\hspace{-28mm}\leq |H_1|+|H_2|-\frac{|H_1||H_2|}{|G|}\,.$$If we denote $n_i=\frac{|H_i|}{|G|}\,$, then $n_i\in(0,\frac{1}{2}]$, and the above inequality leads to
$$\frac{|H_1\cup H_2|}{|G|}\leq n_1+n_2-n_1n_2=n_1(1-n_2)+n_2$$
$$\hspace{25mm}\leq\frac{1-n_2}{2}+n_2=\frac{1+n_2}{2}\leq\frac{1+\frac{1}{2}}{2}=\frac{3}{4}\,,$$as desired.

Clearly, the equality holds if and only if $n_1=n_2=\frac{1}{2}\,$ and $H_1H_2=G$, that is $H_1$ and $H_2$ are distinct maximal subgroups of index $2$. Note that in this case we have $\sigma(G)=3$ by Theorem 1.1, b).
\end{proof}

\subsection{The case $k=3$}

In this case we can take $c_3=\frac{5}{6}\,$, as shows the following theorem.

\begin{theorem}
If $G$ is a finite group with $\sigma(G)\neq 3$, then $|H_1\cup H_2\cup H_3|\leq\frac{5}{6}\,|G|$ for all proper subgroups $H_1$, $H_2$ and $H_3$ of $G$. Moreover, the equality holds if and only if $H_1$, $H_2$, $H_3$ are distinct maximal subgroups, two of index $3$ and one of index $2$.
\end{theorem}

\begin{proof}
Let $H_1$, $H_2$ and $H_3$ be three proper subgroups of $G$ and $n_i=\frac{|H_i|}{|G|}\,$, $i=1,2,3$. Then $n_i\in(0,\frac{1}{2}]$, $i=1,2,3$, and we can assume that $n_1\leq n_2\leq n_3$. One obtains
$$|H_1\cup H_2\cup H_3|=|(H_1\setminus H_3)\cup (H_2\setminus H_3)\cup H_3|=|(H_1\setminus H_3)\cup (H_2\setminus H_3)|+|H_3|$$
$$\hspace{6mm}=|H_1\setminus H_3|+|H_2\setminus H_3|-|(H_1\cap H_2)\setminus H_3|+|H_3|$$
$$\hspace{-28,5mm}\leq |H_1\setminus H_3|+|H_2\setminus H_3|+|H_3|$$
$$\hspace{-1,7mm}=|H_1|-|H_1\cap H_3|+|H_2|-|H_2\cap H_3|+|H_3|$$
$$\hspace{-4,4mm}=|H_1|-\frac{|H_1||H_3|}{|H_1H_3|}+|H_2|-\frac{|H_2||H_3|}{|H_2H_3|}+|H_3|$$
$$\hspace{-3mm}\leq |H_1|-\frac{|H_1||H_3|}{|G|}+|H_2|-\frac{|H_2||H_3|}{|G|}+|H_3|,$$which leads to
$$\frac{|H_1\cup H_2\cup H_3|}{|G|}\leq n_1+n_2+n_3-n_1n_3-n_2n_3.$$Observe that we cannot have $n_2=n_3=\frac{1}{2}\,$ since this would imply $G/H_2\cap H_3\cong C_2\times C_2$, that is $\sigma(G)=3$, a contradiction. So, we can assume that $n_2<\frac{1}{2}\,$. We infer that if $n_3=\frac{1}{2}\,$ then
$$n_1+n_2+n_3-n_1n_3-n_2n_3=\frac{1+n_1+n_2}{2}\leq\frac{1+\frac{2}{3}}{2}=\frac{5}{6}\,,$$while if $n_3<\frac{1}{2}\,$, i.e. $n_3\leq\frac{1}{3}\,$, then
$$n_1+n_2+n_3-n_1n_3-n_2n_3=(1-n_1-n_2)n_3+n_1+n_2\leq(1-n_1-n_2)\frac{1}{3}+n_1+n_2$$
$$\hspace{24,5mm}=\frac{1}{3}+\frac{2}{3}(n_1+n_2)\leq\frac{1}{3}+\frac{2}{3}\,\frac{2}{3}=\frac{7}{9}<\frac{5}{6}\,.$$Consequently, in both cases we have
$$\frac{|H_1\cup H_2\cup H_3|}{|G|}\leq\frac{5}{6}\,,$$as desired.

We remark that the equality holds if and only if $H_1\cap H_2\subseteq H_3$, $H_1H_3=H_2H_3=G$, $n_1=n_2=\frac{1}{3}\,$ and $n_3=\frac{1}{2}\,$, that is $H_1$, $H_2$, $H_3$ are distinct maximal subgroups, two of index $3$ and one of index $2$. Note that in this case we have $\sigma(G)=4$ by Theorem 1.1, c).
\end{proof}

The above proof also shows that for finite groups of odd order the constant $\frac{5}{6}\,$ can be replaced with $\frac{7}{9}\,$.

\begin{corollary}
If $G$ is a finite group of odd order with $\sigma(G)\neq 3$, then $|H_1\cup H_2\cup H_3|\leq\frac{7}{9}\,|G|$ for all proper subgroups $H_1$, $H_2$ and $H_3$ of $G$. Moreover, the equality holds if and only if $H_1$, $H_2$ and $H_3$ are distinct maximal subgroups of index $3$.
\end{corollary}

\subsection{An open problem}

We end this note by pointing out that our problem remains open for an arbitrary $k$.\newpage

\bigskip\noindent{\bf Open problem.} Let $k\geq 4$ be a positive integer and $G$ be a finite group that cannot be written as the union of $k$ proper subgroups. Does exist a constant $c_k\in (0,1)$ such that $|\cup_{i=1}^k H_i|\leq c_k|G|$, for all proper subgroups $H_1$, ..., $H_k$ of $G$? If affirmative, when the equality holds?
\bigskip

Notice that in this case we obtained
$$\frac{|\cup_{i=1}^k H_i|}{|G|}\leq\sum_{i=1}^k n_i-n_k\sum_{i=1}^{k-1}n_i,\0(*)$$where $n_i=\frac{|H_i|}{|G|}\,$, $i=1,2,...,k$, but we failed in giving an upper bound for the right side of $(*)$.
\bigskip

Finally, inspired by the results in the cases $k=2$ and $k=3$, we conjecture that $|\cup_{i=1}^k H_i|$ is maxim when $\sigma(G)=k+1$ and there is a maximal subgroup $M$ of $G$ such that $G=M\cup(\cup_{i=1}^k H_i)$ and $|M|\leq|H_i|,\,\forall\, i=1,...,k$.

\vspace*{5ex}\small

\hfill
\begin{minipage}[t]{5cm}
Marius T\u arn\u auceanu \\
Faculty of  Mathematics \\
``Al.I. Cuza'' University \\
Ia\c si, Romania \\
e-mail: {\tt tarnauc@uaic.ro}
\end{minipage}

\end{document}